\documentclass{amsart}
\usepackage{amssymb}
\usepackage{graphicx}
\usepackage{subfigure}
\pdfoutput=1 
\usepackage{mathrsfs}

\begin{document}

\newtheorem{theorem}{Theorem}[section]
\newtheorem{lemma}[theorem]{Lemma}
\newtheorem{proposition}[theorem]{Proposition}
\newtheorem{corollary}[theorem]{Corollary}
\newtheorem{conjecture}[theorem]{Conjecture}
\theoremstyle{definition}
\newtheorem{remarks}[theorem]{Remarks}
\newtheorem{note}[theorem]{Remark}
\newtheorem{definition}[theorem]{Definition}

\newenvironment{example}[1][Example]{\begin{trivlist}
\item[\hskip \labelsep {\bfseries #1}]}{\end{trivlist}}
\newenvironment{remark}[1][Remark.]{\begin{trivlist}
\item[\hskip \labelsep {\bfseries #1}]}{\end{trivlist}}

\title{Billiards With Bombs}
\date{May 3, 2015}
\author{Edward Newkirk, Brown University}

\begin{abstract}
In this paper, we define a variant of billiards in which the ball bounces around a square grid erasing walls as it goes. We prove that there exist periodic tunnels with arbitrarily large period from any possible starting point, that there exist non-periodic tunnels from any possible starting point, and that there are versions of the problem for which the same starting point and initial direction result in periodic tunnels of arbitrarily large period. We conjecture that there exist starting conditions which do not lead to tunnels, justify the conjecture with simulation evidence, and discuss the difficulty of proving it.
\end{abstract}

\maketitle

\begin{centering}
\section{Introduction}
\end{centering}

The classic arcade game Breakout involves using a paddle to bat a ball against a pile of bricks which are erased when the ball bounces off them, in an attempt to clear away all the bricks before the ball can run off the screen. Mathematicians have studied bouncing balls in several contexts, most notably in connection with billiards; one  might wonder what would happen if the trick of erasing hit obstacles were brought to such study from Breakout. The most Breakout-like model, where the ball erases the 2-skeleton of some initial array of obstacles, has been studied by Xavier Bressaud and Marie-Claire Fournier \cite{bressaud}; this paper will attempt to expand the idea of erasing obstacles into a wider set of problems, allowing the ball to erase the 1-skeleton of a square grid in a variety of patterns.

We refer to these patterns as "bombs". A bomb is defined as a pattern of walls with one particular wall marked and one side of the marked wall chosen. Every time the ball hits a wall, it erases a pattern of walls corresponding to the bomb, rotated in such a way that the ball is hitting the marked wall from the chosen side. An example of a slightly more complicated bomb is in Figure \ref{wingedwedgesample}.

We define a dynamical system by setting a particle down in the unit square, in a plane covered by a unit square grid; giving it an initial direction of motion, which we'll typically refer to in terms of the slope; defining a bomb; and allowing it to reflect off walls as in conventional billiards, erasing walls in the bomb's shape as it goes. We study the behavior of the particle as it runs off to infinity, investigating which combinations of bomb, starting position, and starting direction result in predictable behavior. 

Since the ball erases every wall it hits, it cannot settle into a periodic orbit as it sometimes does in conventional billiards, but it can and does dig periodic tunnels where (say) every six collisions it moves one square up and one to the right, with each collision corresponding to the sixth-previous collision displaced by $(1,1)$ (see figure \ref{threepics}). There are also cases where the ball seems to dig in one or more directions without ever settling into a perfectly periodic pattern, so we define tunneling behavior slightly more broadly:  

\begin{definition}
\label{tunneldefs}
A \emph{band} is the set of points in the plane within a given finite distance $\epsilon$ of some ray $\ell$. We say a particle \emph{tunnels} for a given starting point and direction if, when the particle is allowed to run infinitely from those starting conditions, every wall collided with outside a disc of radius $r$ is contained in a finite union of bands. A particle whose infinite orbit requires two bands no matter what finite value of $r$ is used is said to create a \emph{bidirectional} tunnel. A particle digs a \emph{periodic} tunnel if, for sufficiently large $r$, there exists some integer $k$ such that the wall hit in collision $n+k$ is always of the same type as the wall hit in collision $n$ and the displacement from $n$ to $n+k$ depends only on which band contains wall $n$. The least such $k$ is the \emph{period} of the tunnel.
\end{definition}

If the ball digs periodically in some particular direction, it will clearly stay within bounded distance of a ray drawn through two separate instances of the same step in the period, so our initial example of periodic tunneling does satisfy the formal definition. We permit the displacement of a periodic tunnel to vary by band because that's necessary for bidirectional tunnels - the displacement along a given band is obviously in the direction of the ray defining the band, so if there are bands in multiple directions then we will have multiple possible displacements. It will turn out that all tunnels are at most bidirectional, and that if a tunnel is bidirectional then its two directions are directly opposite (Lemma \ref{twodirections}), so when a particle digs a tunnel we can sensibly talk about the slope of the tunnel.

The majority of our time will be spent studying the simplest bomb shape possible, where the particle erases only the walls it hits; as it turns out, we can find a vast number of tunnels with this bomb even looking at a relatively narrow set of directions.

\begin{theorem}
\label{headlinethm}
For every point $P$ in the unit square, there exists some $\epsilon(P) \leq \frac{1}{17}$ such that for all but countably many slopes $s \in [3, 3+\epsilon(P)]$ a particle starting from $P$ with the single-wall bomb and slope $s$ will tunnel, and the tunnel will have slope $1 + \frac{3s-9}{25-8s}$.. If $s$ is rational, any tunnel will be periodic. If $s$ is irrational, the tunnel will not be periodic; there are therefore uncountably many aperiodic tunnels from every starting point in the square. (Proved as Theorem \ref{everypointtunnels} and Corollaries \ref{periodiconcestarted}, \ref{aperiodiconcestarted}, and \ref{uncountabletunnels}).
\end{theorem}

While the single-wall bomb is the most straightforward, it is not the only one we consider. We'll look at entire families of triangular-wedge-shaped bombs, and see that

\begin{theorem}
\label{bombsheadline}
Let $q \geq 6$ be an even integer. There exist infinitely many bombs for which a particle starting from the left-hand wall with slope $\pm 3$ tunnels with period $q$. (Theorem \ref{s3n1mod3} defines the bombs explicitly).
\end{theorem}

The paper is organized as follows:

\subsection*{Section 2: Definitions And Conventions.}
Section 2 will define billiards with bombs formally on a square grid and introduce the standard language we'll use to discuss it. We also prove the claim from this introduction that tunnels are at most bidirectional, and give a formal definition of the slope of a tunnel. 

\subsection*{Section 3: Cutting Sequences.}
We describe how we can use previous research into cutting sequences on the square torus, explaining the basis of the simulations which inspired our conjectures and the tools we'll be using to prove (some of) those conjectures.

\subsection*{Section 4: An Infinite Set Of Slopes Giving Periodic Tunnels.}
Sections 4-8 all deal with the case where the ball erases a single wall, which is the simplest case and the one which has been studied in the most detail. In Section 4 we describe the simplest kind of periodic tunneling that the Breakout system might display and show that there is a set of starting conditions, with infinitely many integer slopes, for which the system produces such periodic tunnels. Our argument only works for integer slopes and fractional slopes with denominator 3; we'll expand the number of periodic tunnels later in the paper.

\subsection*{Section 5: Tunneling With Reorganization.}
Slopes sufficiently close to the ones corresponding to the periodic tunnels from Section 4 often (but not always) seem to imitate those periodic tunnels with occasional reorganizational hiccups. In Section 5, we display several examples of this behavior, one example of a case where it doesn't seem to happen, and then work through one specific example by hand to show that, if they get started, slopes in the interval $(3, 3\frac{1}{17}]$ do produce tunnels-with-reorganization. We also observe that, if they start tunneling, rational slopes in that interval must produce periodic tunnels and irrational slopes cannot. 

\subsection*{Section 6: Continuity And Uncountable Tunneling Slopes.}
In Section 6, we discuss the ways in which close starting conditions produce similar behavior, effectively defining continuity for the Breakout problem, and use this knowledge to show that the tunneling-with-reorganization behavior from Section 5 does indeed get the start it needs for slopes sufficiently close to 3 (the exact distance depends on the starting point).

\subsection*{Section 7: Slope 146 And Delayed Tunneling.}
In Section 7, we discuss the most spectacular example we've found so far of starting conditions which take a long time to start tunneling. A particle with slope 146 will always eventually wind up digging a horizontal tunnel, but the conditions it needs to begin tunneling are sufficiently rare that it can take millions of steps to begin doing so. We'll describe the tunnel it eventually digs and list the time required to get there from every possible starting position.

\subsection*{Section 8: Non-Tunneling Behavior.}
In Section 8, we speculate as to what might happen when the particle fails to tunnel, and present a specific set of starting conditions from which we conjecture the particle clears the plane.

\subsection*{Section 9: Wedge-shaped Bombs.}
In Section 9, we discuss how the ball seems to behave when it erases a triangular wedge, conjecturing that for any integer slope there are infinitely many wedge sizes leading it to tunnel. We prove that for slope 3 there are infinitely many wedges leading to periodic tunnels with any even period $\geq 6$.

\subsection*{Section 10: Open Problems.}
In Section 10, we summarize the problems left open by sections 2-9, and briefly discuss further variants of the problem.

\subsection*{Appendix: Raw Integer Slope Simulation Data.}
We've run broad tests of integer slopes for a few different bomb shapes while trying to build intuition for the problem's behavior. This appendix links to google docs of the test results while briefly summarizing them.

\begin{centering}
\section*{Acknowledgements}
\end{centering}
The author thanks Xavier Bressaud for beginning the mathematical study of Breakout and Nicolas Bedaride for suggesting the 1-skeleton variant from which this paper grew. This work was done as part of a doctoral dissertation under the supervision of Richard Schwartz at Brown University; the author thanks Professor Schwartz for his advice and encouragement throughout this process, and the Brown math department as a whole for funding and a very congenial working environment.

\begin{centering}
\section{Definitions And Conventions}
\end{centering}

Starting with a unit square grid covering the plane, we create a dynamical system by creating a particle with the following four qualities:

\begin{itemize}
\item A starting point in the plane.
\item A starting direction of motion.
\item A choice of which edges on the square grid start out solid.
\item A pattern of edges on the square grid consisting of one marked horizontal edge (the one the particle will hit) and a finite (possibly 0) number of other edges, which may be horizontal or vertical. This pattern determines which walls will be erased, and will be referred to henceforward as the \emph{bomb}. 
\end{itemize}

The particle moves according to the following rules:

\begin{enumerate}
\item The particle always moves in a straight line until it collides with a solid wall.
\item When the particle hits a wall, it reflects off, preserving the angle of incidence, and erases the walls given by rotating the bomb so that the marked edge overlays the wall just hit and the particle is hitting the marked edge from below relative to the bomb.
\item If the particle hits a corner, it stops, whether the walls are still there or not. These cases are less interesting, so they are not our main focus. 
\end{enumerate}

\begin{figure}[h]
\begin{center}
\includegraphics[height=4 cm]{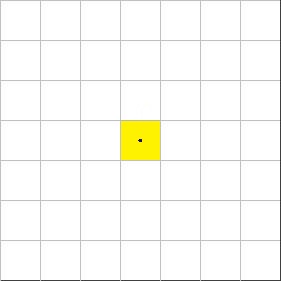}
\includegraphics[height=4 cm]{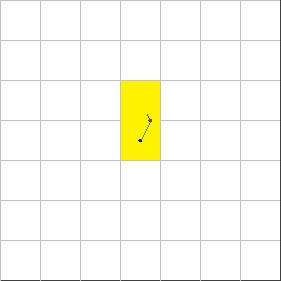}
\includegraphics[height=4 cm]{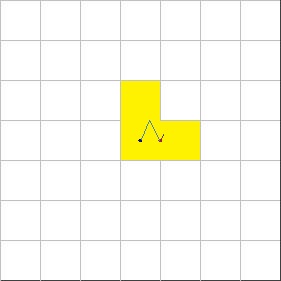}
\end{center}
\caption{\label{bouncepicture}
A particle's first two collisions after starting from the center of the square with initial direction $\left< \frac{1}{\sqrt{5}},\frac{2}{\sqrt{5}} \right>$ (equivalently, with intial slope 2). In this and subsequent simulated images, the starting square is marked by a blue dot and the red tail behind the particle displays its direction. Here, we've also tracked the particle's path in blue.
}
\end{figure}

\begin{remark}
\label{definitionchoices}
Every case discussed in this paper begins with every wall in the plane solid. We include the choice to do otherwise in our definition of the problem because a state of this dynamical system consists of the ball's location, the ball's direction, and the solidity of every individual wall; although the initial solidity is not something we vary here, it is part of the starting state and should therefore be chosen explicitly. Although we won't directly discuss it much, the state space is $R^2\times S^1 \times 2^Z$, with the system moving by horizontal/vertical reflection in the $S^1$ part (the particle's direction) and decreasing in $2^Z$ as walls are erased but never come back.

Another part of the definition worth discussing is the choice of isometries by which the bomb is laid over a hit wall. We only allow rotation, but it would be equally valid to redefine the system so that the bomb is oriented by rotation and reflection (i.e. rotate and reflect as if the particle is hitting the marked wall from below, moving up-right, with some convention for how to handle perfectly horizontal or vertical movement). Every bomb discussed in this paper is symmetric under reflection, so there's no practical distinction and we use the simpler definition, but there are bomb shapes for which reflection or lack thereof could make a big difference. 
\end{remark}

We see from the above that, since we're fixing the starting wall configuration, the system is completely determined by its starting point $(x,y)$, its starting direction, and the bomb shape. Translating the starting point by an element of $Z \times Z$ will translate the resulting open region by the same element with no other changes, so we can assume that the particle starts moving into the square $(0,1) \times (0,1)$. Similarly, reflecting the starting point and direction across the lines $x=\frac{1}{2}$, $y=\frac{1}{2}$, and $x=y$ will reflect the particle's future path and hence the open region across the same line with no other changes. By tolerating such reflections and ignoring the trivial cases where the particle's motion is perfectly horizontal or vertical, we can assume that the particle is heading up and to the right with slope $s>1$ and starting location in $[0,1) \times [0,1)$.

As the particle runs, it will clear an open region including its starting point; the majority of this paper will be devoted to studying what sorts of open region can (and do) result as the particle runs to infinity. Usually, the region can be described as a blob (so largely resisting formal definition) or as a tunnel in some specific direction(s). We've already defined what it means for a particle to tunnel, but we can almost immediately narrow down what sorts of tunnels can possibly occur. 

The most basic region that a particle might clear is a two-directional tunnel. Starting with a completely horizontal direction of movement, or completely vertical, will trivially result in the particle bouncing back and forth to create a horizontal (vertical) bidirectional tunnel, whose exact shape depends on the chosen bomb. As Figure \ref{slope1} shows, we see the same behavior with slope 1 for the single-wall bomb; in fact, when testing different bomb shapes, this sort of diagonal tunnel is by far the most common outcome from slope 1.
 
 \begin{figure}[h]
\begin{center}
\includegraphics[height=6 cm]{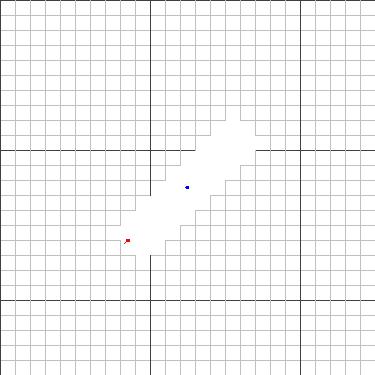}
\includegraphics[height=6 cm]{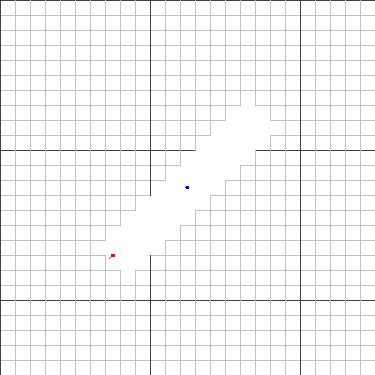}
\end{center}
\caption{\label{slope1}
Starting the particle with slope 1 and the most basic bomb shape (a single wall) results in a diagonal tunnel expanding in both directions with period 16; the particle advanced one full period between the two pictures. 
}
\end{figure}

The particle's starting place is actually irrelevant in this case: slope 1 gives a cutting sequence of period 2, so there are only two distinct starting areas, which are reflections of each other across the line $y=x$, which reflection preserves slope 1 and the resulting tunnel.

All three of these examples see the particle digging in two directions which are directly opposite one another and also happen to be $\pm$ the particle's initial direction of movement. It turns out that this must always be the case if the particle is to tunnel in more than one direction. The following theorem and its proof are directly inspired by (\cite{bressaud}, Lemma 4.4), which states the analogous result for the variant where the particle erases square bricks rather than one-dimensional walls; similarly, our definition of tunnels is a close paraphrase of his. 

\begin{lemma}
\label{twodirections}
A particle which tunnels with slope $s$ can have its limiting path contained in at most two bands. If two bands are required, then they will be directly opposite one another, with opposite directions, and the slope of each band will be equal to $s$.
\end{lemma}
\begin{proof}
Suppose the particle's path has to be contained in at least two bands, meaning that there's no point after which it sticks to one band. Then there must be some pair of bands $B_1,B_2$, with directions given by the rays $\ell_1, \ell_2$, such that the particle goes directly from band $B_1$ to band $B_2$ infinitely often, passing through the disc of radius $r$ on its way. As there are only finitely many walls inside the disc, eventually the particle has to go directly from $B_1$ to $B_2$ without changing direction. This means that the entrances from the disc to the parts of the tunnel associated with each band have to be directly opposite one another. It also means that the line segment between where the particle pauses work in $B_1$ and where it starts work in $B_2$ has to be some reflection of the ball's initial direction of movement; in particular, it has slope $s$. However, as time goes on, the particle's path from the leading edge of the $B_1$ tunnel back to the disc gets longer and longer, so the angle between the particle's exit direction and the direction of $\ell_1$ gets smaller and smaller; similarly, the angle between its entry to the $B_2$ tunnel and the direction of $\ell_2$ gets smaller and smaller.  Since the particle is heading in the same direction in both cases, $\ell_1$ and $\ell_2$ are forced to make an angle of $\pi$ radians when appropriately translated and must have a direction which is some reflection of the ball's. Since every band which gets entered infinitely often has an angle of $\pi$ with its predecessor, it is impossible to have more than two bands.
\end{proof}

Since tunnels are either unidirectional or bidirectional in directly opposite directions, it makes sense to talk about the singular direction of a tunnel. Once we start finding specific periodic tunnels, it'll be particularly convenient to talk about the slope of a tunnel, because the same periodic behavior can carry a particle (1,2), (1,-2), (-1,2), or (-1,-2), depending on whether the particle was headed up or down and left or right when it started its cycle. Those four displacements will be digging in four different directions, but the tunnels have the same band up to translation and reflection - and, for some of our results, up to translation and reflection will be the most precision we can achieve. We therefore say that:

\begin{definition}
\label{tunnelslopedef}
The \emph{slope} of a tunnel is the absolute value of the slope of the ray defining one of its bands.
\end{definition}

So, in our hypothetical, all four tunnels have slope 2, with one going up-right, one going up-left, one going down-right, and one going down-left. This does unfortunately mean that we're using "slope" to discuss a property of every starting state and a different property of (one category of) outcome. We have attempted to specify which is under discussion whenever it is not clear from context, and apologize for any remaining confusion.

An open question is whether there are any bidirectional tunnels besides the trivial ones and the particle-slope-1 tunnel. None have been found in simulation; however, related variants of the problem often see bidirectional tunnels not associated with particle slope 1 (see Section 10 for further details).

\begin{conjecture}
\label{bidirectionalconjecture}
The only bidirectional tunnels for the simplest bomb shape (where only the wall hit is erased) are those corresponding to an initial angle of movement which is an integer multiple of $\frac{\pi}{4}$.
\end{conjecture}

\begin{centering}
\section{Cutting Sequences}
\end{centering}

Once the particle is moving, we care a great deal about which walls it hits, but less so about where on the wall it hits. In standard billiards, this manifests itself in the study of cutting sequences, which simply list the walls hit by a particle as it bounces around. We can make use of previous research into cutting sequences by looking at our particle's behavior in the right way. Our particle's tendendency to destroy walls makes using cutting sequences more complicated, but once we manage to connect them to this variant of the problem they will be a very convenient computational tool. This section is focused on doing so.

\begin{definition}
\label{encounterdef}
An \emph{encounter} is when the particle reaches a piece of the plane which was originally a wall (and may or may not still be a wall), or equivalently when the particle attains an integer $x-$ or $y-$ value. A \emph{collision} is when the particle bounces off a wall and erases the wall; we note that the collisions of a given particle are a strict subset of its encounters. Collisions and encounters will be labeled \emph{horizontal} or \emph{vertical} according to the orientation of the wall in question.
\end{definition}

Looking at the particle's progress in terms of encounters rather than collisions has the obvious disadvantage that sometimes we check in on the particle and find that it's sailing through open space not changing direction or erasing anything. That disadvantage is outweighed, however, by the fact that encounters behave in ways which have already been researched: 

\begin{lemma}
\label{cuttingsequence}
The sequence of encounters associated with a given starting location and slope, expressed as horizontal (which we'll often abbreviate to H) or vertical (V), is equivalent to the cutting sequence associated to the same starting location and slope on the infinite square grid.
\end{lemma}
\begin{proof}
If the particle bounces off a wall, its new direction corresponds to reflecting the path given by its previous direction across a line of the form $x=n$ or $y=n$, $n$ some integer. Such reflections are symmetries of the original grid, and accordingly they map prospective horizontal encounters to horizontal encounters and vertical to vertical. The particle's sequence of encounters will therefore follow the same pattern as if it had been allowed to run forever without bouncing off anything.
\end{proof}

The infinite cutting sequence associated with a given slope is a Sturmian sequence with the spacing between "V"s directly related to the slope's continued fraction; for specific details of the relationship, see, e.g., \cite{cuttingsequence}.

\begin{remark}
\label{simulatorexplanation}
Since cutting sequences on the square torus can be readily derived from the partial fraction expansion of the slope, this correspondance makes computation much easier: for a given starting position and slope, it's easy to determine whether horizontal encounter $m$ comes before or after vertical encounter $n$, and in particular if the position and slope are both rational then the type of the next encounter can be determined by simple integer calculations. Every simulation used for pictures or to inspire conjectures in this paper was obtained from a simulator using such calculations to figure out the type of the next encounter, usually with starting position $(0,\frac{1}{2})$ or $(\frac{1}{2},\frac{1}{2})$. 
\end{remark}

Cutting sequences give us another way to define starting positions for a given bomb: take a Sturmian sequence, pick a step to start at, and continue with the appropriate sequence of encounters from there. Thinking about starting positions this way lets us determine which geometric starting positions are genuinely different - that is, which starting positions will eventually lead to different encounters and/or collisions. It's clear that different cutting sequences will eventually diverge, so different slopes always mean genuinely different starting conditions. With the same cutting sequence, though, the only possible difference is where in the sequence we start. Two different starting points which correspond to the same place in the cutting sequence will give the same result.

For irrational slopes, this just means that sliding the starting point forwards or backwards along a line with appropriate slope won't make a difference. The locations where a particle with irrational slope encounters horizontal (or vertical) walls are dense on the unit interval, so any movement not on that line must be large enough to change some eventual encounter. For rational slopes, however, the continued fraction expansion terminates, so the cutting sequence is periodic and we have only finitely many meaningfully distinct starting states. While moving the particle's starting position, we cannot change an encounter without crossing over a corner, so the boundaries between the distinct starting states must be the set of points from which the particle eventually hits a corner. This set is easy to find; just look at the line on the square torus starting at $(0,0)$ and carrying on to $(1,1)$ with appropriate slope.

\begin{figure}[h]
\begin{center}
\includegraphics[height=6 cm]{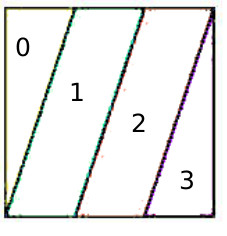}
\end{center}
\caption{\label{3startspic}
A particle with slope 3 has cutting sequence $\overline{HHHV}$, so will have four meaningfully distinct starting positions; they're labeled 0 through 3 with 0 being immediately after a vertical encounter, 1 after a single horizontal encounter, 2 after two horizontal encounters, and 3 after three vertical encounters.
}
\end{figure}

When talking in terms of cutting sequences, there isn't a well-defined path for the particle to travel, so the system effectively becomes a discrete dynamical system with the particle moving from one encounter to the next according to the cutting sequence. Since we can go back and forth between Sturmian sequences and directions of travel, we can move between the continuous version of the problem and the discrete version more or less at will, without changing which walls the particle hits at all - effectively we just cut all the states where the particle isn't on an integer x- or y- value out of the continuous version's state space. With a periodic cutting sequence it's even possible to describe the system as a cellular automaton, although we will omit the details since no benefit to doing so has been found.

\begin{centering}
\section{An Unbounded Set Of Slopes Giving Periodic Tunnels}
\end{centering}

We now have the tools we need to begin looking for tunnels. This section will start our search by constructing an infinite sequence of integer slopes all of which dig very simple periodic tunnels for at least one starting location. Once we understand how these simple tunnels work, our path to Theorem \ref{headlinethm} will involve finding variants of the easiest example. We will assume until section 9 that the bomb consists of only the single wall which the particle hits. \\

\begin{definition}
\label{freshdef}
In the next few sections, we will use \emph{column} to refer to the vertical strip $[n,n+1]\times \infty$, $n$ an integer, whose left and right borders are made up of vertical walls and which contains countably many horizontal walls inside. A \emph{fresh} column will be a column with all its horizontal walls still present and all the vertical walls on either side still present; when we talk about the particle entering a fresh column, we allow for one vertical hole which the particle passes through to enter the column.
\end{definition}

We begin by thinking about what a very simple periodic tunnel might look like once its periodic behavior had started. Assuming that the particle is headed broadly to the right, the simplest behavior imaginable would be for it to enter a fresh column from the left, bounce its way over to the right, bounce back to the left eventually hitting the left side of the column at a different wall than it entered by, and then bounce back over to the right and exit through the same wall it hit the first time. It will then be entering another fresh column; if it encounters walls in the same sequence, we get the same outcome, and so on unto eternity. To get the same sequence in every fresh column, we need a cutting sequence which is periodic with a period that contains either one or three vertical encounters; this corresponds to slopes of the form $s=n$ or $s=\frac{n}{3}$. The next step, then, is to find a slope of that description which also reacts appropriately to a fresh column.

The key to doing so is to notice that, as long as we're still in the same column, the sequence of horizontal walls encountered does not depend on when exactly the particle has its vertical encounters; vertical encounters do not change the height of the next horizontal encounter, and since we are staying in one column our previous horizontal encounters completely determine whether the next encounter is with a solid wall or not. We therefore calculate the sequence of horizontal encounters, with the idea of eventually dropping vertical encounters in at convenient points. When we do insert vertical encounters, they will come on the sides of whichever square the particle happens to be entering after its last horizontal encounter.

By reflection and translation, we suppose that the particle enters a fresh column by the left-hand-wall of the square $[0,1]\times[0,1]$, heading up and to the right. Throughout these calculations, squares and vertical walls will be labeled by the height given by their lower bound in standard coordinates.  After the particle's first horizontal encounter, it will be headed down and into square 0, as it collides with the horizontal wall at height 1. The second horizontal encounter will see it collide with the horizontal wall at height 0 and head up into square 0. The third horizontal encounter sees it pass through horizontal wall 1, heading up into square 1; the fourth sees it bounce off horizontal wall 2 and head down into square 1; the fifth sees it pass through wall 1 and head down into square 0.

Just by doing this much, we can see what will happen when a particle with slope $5/3$ enters a fresh column. It will hit a vertical wall after one or two horizontal encounters; in either case it is heading into square 0, so it hits the right-hand vertical wall at height 0. It hits its second vertical wall after three or four encounters, in either case the left-hand wall at height 1. Finally, it hits the third vertical wall after five encounters, heading down into square 0, so will pass through the wall it hit previously and head into another fresh column with its vertical direction flipped. Sure enough, in simulations with slope $5/3$, we see a periodic tunnel whose trail flips vertically with every column it passes through:

\begin{figure}[h]
\begin{center}
\includegraphics[height=5 cm]{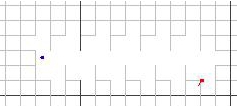}
\end{center}
\caption{\label{fivethirdspic}
A periodic tunnel produced by starting with slope $5/3$ from the left-hand side of the starting square. 
}
\end{figure}

Calculating further, we see that the third time the particle heads down into square 0 is after its thirteenth horizontal encounter; the number of encounters between heading down into square 0 has gone from four to eight, an increment of four. This turns out to represent a significant pattern:

\begin{lemma} 
\label{quadraticencounters}
Let $d(n)_k$ be the sequence of integers indexed by $k$ such that the $k$th time the particle heads down into square $n$ of a fresh column is after its $d(n)_k$th horizontal encounter; for instance, $d(0)_k$ = $(1,5,13,\ldots)$. Then there exist integers $b,c$ such that $d(n)_k = 2k^2 + bk + c$. The same holds for the sequence $u(n)_k$ giving the times when the particle heads up into square $n$.
\end{lemma}
\begin{proof}
Showing that a sequence of integers is given by a quadratic function with lead term $2k^2$ is equivalent to showing that the difference between consecutive terms increases by $4$ every time. The difference between consecutive terms of $d(n)_k$ is the number of horizontal encounters which happen before the particle heads down into square $n$ again. What happens during that time is that the particle passes through all the cleared walls below square $n$ and hits the first solid wall beneath it; passes back up through all the previously cleared walls below square $n$ until it's heading up into square $n$; passes up through all the cleared walls above square $n$ until it hits the first solid wall; and passes back down through all the previously cleared walls above square $n$ until it comes back down into square $n$. Each time the particle does this, it clears a new wall above square $n$ and a new wall below, with the result that each stage of passing through all the previously cleared walls takes one more encounter for a total of exactly $4$ more encounters. The proof for $u(n)_k$ is analogous.
\end{proof}

This gives us the tools we need to calculate when the particle enters some specific rows:

\begin{corollary}
\label{niceslots}
Let $n$ be the number of horizontal encounters a particle has had in a fresh column which it entered heading upwards.
\begin{itemize}
\item If $n = 2x^2 - 2x + 1$ for some positive integer $x$, the particle is heading down into square 0.
\item If $n = 2x^2$, the particle is heading up into square 0.
\item If $n = 2x^2+1$, the particle is heading up into square 1.
\end{itemize}
\end{corollary}
\begin{proof}
We know by Lemma \ref{quadraticencounters} that each of these cases is given by $n=2x^2+bx+c$ for some integer $b,c$. Using the convention that $x=1$ should give the first time each combination of square and direction happens, we only need to compute until the second time it happens to have a system of two equations and two unknowns which is easily solved.
\end{proof}

One immediate consequence of this is that if we have an integer slope $s$, it hits the left-hand side of the column after $2s$ horizontal encounters, so is guaranteed to hit a solid wall the first time unless $s=x^2$ for some s. It will also let us find slopes which create periodic tunnels as follows:

\begin{lemma}
\label{unboundedsequence}
Let the three integer sequences $s_n, x_n, y_n$ be given by $s_1 = 3, s_{n+1} = 6s_n^2-8s_n+3$ and $x_1 = 1, y_1 = 2, x_{n+1} = 2 x_n y_n, y_{n+1} = 6x_n^2+1$. Then for all $n$, $s_n = 2 x_n^2+1$ and $3*s_n = 2 y_n^2 + 1$.
\end{lemma}
\begin{proof}
For $n=1$ this is trivial, and it follows for larger $n$ by induction. Supposing it is true for $s_n$ then some straightforward algebra gives
\begin{align*}
2x_{n+1}^2+1 &=  s_{n+1}.\\
\end{align*}
It's also fairly easy to see that 
\begin{align*}
y_{n+1}^2 &= 3 x_{n+1}^2+1 ,
\end{align*}
which gives the desired result for $y_{n+1}$ when plugged into the $x_{n+1}$ equation.
\end{proof}

We are now ready to prove the existance of countably many distinct starting conditions producing periodic tunnels:

\begin{theorem}
\label{unboundedperiodic}
Again let $s_n$  be the sequence given by $s_1 = 3, s_{n+1} = 6s_n^2-8s_n+3$. Then starting with slope $s_n$ at any point in the interior of the triangle with vertices $(0,0), (0, 1), (\frac{1}{s_n}, 1)$, or at $(0,y)$ for $0<y<1$, produces a periodic tunnel up and to the right. 
\end{theorem}
\begin{proof}
By Corollary \ref{niceslots} and Lemma \ref{unboundedsequence}, if $s_n$ is not a perfect square then whenever it enters a fresh column it bounces up and down until it hits the far vertical wall at height 1 (assuming that it enters at height 0 going up); bounces back across to hit at a solid wall; and bounces back once more to eventually pass through the far horizontal wall at height 1 with no other wall in its new column touched. 

\begin{figure}[h]
\begin{center}
\includegraphics[height=6 cm]{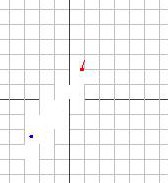}
\includegraphics[height=6 cm]{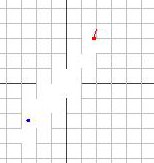}
\end{center}
\caption{\label{threepics}
The first slope in this set, $s_1 = 3$, produces the above tunnel with a period of six collisions; the image on the right shows what happens after advancing one complete period.
}
\end{figure}

If every column in the horizontal direction it's moving is fresh, it will keep doing this forever, producing a periodic tunnel. That condition is trivially true if it starts at the place in its cutting sequence just after the vertical encounter, as the first square is to all relevant purposes a fresh column and nothing it enters after that can possibly have been touched; the region given corresponds to that part of the cutting sequence. Starting with the vertical encounter also works to guarantee this, and will send the particle tunnelling up and to the left. 

The last thing we need to check is that $s_n$ isn't a perfect square. But if $s_n \equiv 3$ (mod 10), then so does  $s_{n+1}$, so by induction none of the $s_i$ are squares.
\end{proof}

Obviously these slopes will behave periodically for many other starting positions - all it needs is that every column they enter after a certain point should be untouched, at least at the heights they reach during the periodic step, or even just have no horizontal walls cleared and one particular vertical wall solid, which is a very likely condition once they start breaking into new parts of the plane. 

\begin{conjecture}
\label{alwaysperiodic}
For $s_n$ as above, any starting position which does not cause the particle to hit a corner will eventually produce a periodic tunnel.
\end{conjecture}

However, there is no obvious way to prove that a given slope reaches the right conditions without checking each starting location by hand, and $s_n$ gets very large very fast - the sequence begins $(3, 33, 6273,\ldots)$. The conjecture has been checked for $s=3$. 

\section{Tunneling With Reorganization}

So far we've only looked at very simple slopes - integers and fractions with denominator 3. The problem is well-defined for any slope, and it's particularly interesting to look at what happens when the slope gets changed slightly. Intuitively, one might expect that small changes would snowball over time to make a large difference, especially small changes to a slope which was producing a periodic tunnel. However, in practice very different behavior can result, as we see in Figure \ref{reorghappens}. The main goal of this section will be to prove Lemma \ref{reorglemma}, which states the exact conditions under which a slope slightly bigger than 3 will behave as shown in Figure \ref{reorghappens}; this lemma will then be the basis for our proof that there exist uncountably many tunneling slopes for every possible starting point.

\begin{figure}[h]
\begin{center}
\includegraphics[height=6 cm]{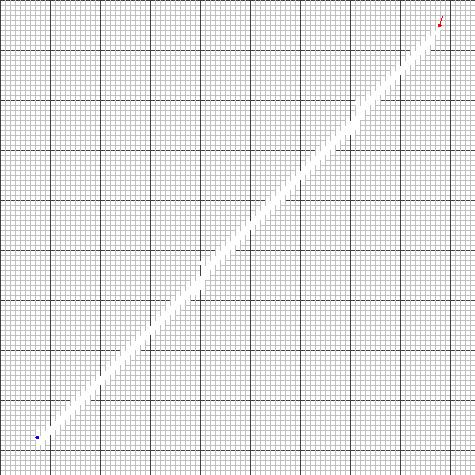}
\includegraphics[height=6 cm]{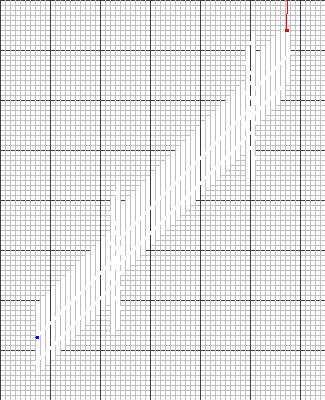}
\end{center}
\caption{\label{reorghappens}
The tunnels produced by slopes 3.01 and 33.01 respectively. These slopes spend a lot of time clearing the same patterns as 3 and 33, but pause every so often to reorganize.
}
\end{figure}

It's fairly easy to see why this might happen. A particle with slope three will have cutting sequence $\overline{HHHV}$, where $H$ corresponds to a horizontal encounter and $V$ to a vertical encounter. A particle with slope $3+\epsilon$ will have a cutting sequence of mostly $HHHV$, but with the occasional $HHHHV$. Given that slope 3 produces a periodic tunnel given some very simple starting conditions, it's not astonishing that 3.01 spends most of its time between instances of $HHHHV$ following the same pattern, and since we know from the partial fraction expansion of 3.01 that those instances are regularly spaced it is plausible that the resulting tunnel should turn out to be periodic. On the other hand, Figure \ref{reorgdoesnthappen} shows that it's not always true that close approximations of tunnelling slopes tunnel.

\begin{figure}[h]
\begin{center}
\includegraphics[height=6 cm]{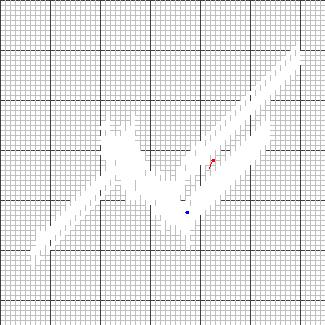}
\includegraphics[height=6 cm]{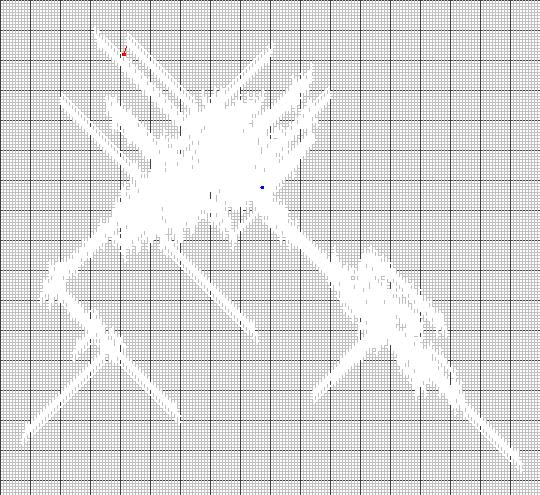}
\end{center}
\caption{\label{reorgdoesnthappen}
Slope 2.99, starting from the left of the square, after one and ten thousand collisions. No obvious overarching pattern.
}
\end{figure}

There does not seem to be an obvious way to tell which minor deviations from a periodic setup will still tunnel and which will become chaotic. What we can say is that the apparent tunnelling behavior of 3.01 is legitimate; the rest of this section will be devoted to articulating what is happening in that image, proving that it continues indefinitely, and finding the largest $\epsilon$ such that $3+\epsilon$ follows the same pattern. 

A quick and informal summary of what we see with slope 3.01 is that, first, when the particle enters a fresh column at a stage of its cutting sequence when it has enough $HHHV$s coming up, it reacts the same way as if it had slope 3, and second, when an $HHHHV$ chunk gets involved, the particle bounces around clearing a larger blob before eventually successfuly reorganizing itself and going back to the slope-3 tunnel for a while. It turns out that this reorganization is guaranteed to succeed as long as it comes after at least one column of slope-3 tunneling, and we can state exactly which slopes space out their $HHHHV$s widely enough for that to happen. 

\begin{lemma}
\label{reorglemma}
Suppose a particle has a cutting sequence corresponding to a slope in the interval $(3,3 \frac{1}{17}]$. Such cutting sequences can be split into chunks of the form $HHHV$ and $HHHHV$ by splitting so that each chunk ends with its unique vertical encounter. \begin{itemize}
\item If the particle enters a fresh column at a point in the cutting sequence where the next three chunks are all $HHHV$, it will pass through that column in the same way as a particle of slope 3. 
\item If it enters a fresh column with one of the next three chunks $HHHHV$ \textbf{and} the column it just left was cleared out in the slope-3 fashion \textbf{and} the next column in whichever direction it is moving is also fresh, the particle will pass through its current column and the next column and enter the column after next without hitting any walls in that column apart from the wall it enters through and with at least three $HHHV$ chunks to follow.
\end{itemize}
\end{lemma}
\begin{proof}
The proof of part 1 is simple: three $HHHV$ chunks immediately after entering a fresh column is what slope 3 uses to do one column's worth of tunneling, so any identical cutting sequence will do the same. The proof of part 2 boils down to tracking by hand what happens if we enter a fresh column with an $HHHHV$ chunk line up. 

 \begin{figure}[h]
\begin{center}
\includegraphics[height=7 cm]{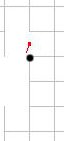}
\end{center}
\caption{\label{setuppic}
A particle entering a fresh column after clearing the previous column in the slope-3 fashion
}
\end{figure}

We begin in the situation depicted by Figure \ref{setuppic}. Throughout this section, walls will be labeled by their orientation (horizontal or vertical) and their lower or left endpoint. We also translate and reflect the plane so that the particle starts out moving up and right into the square $[0,1] \times [0,1]$. This means that the origin is the point marked with a black dot in the picture, the particle is passing through the wall (0,0)V, and the walls cleared in the previous column were (-1,-2)V, (-1,-1)V and the horizontal walls from (-1,-2)H through (-1,1)H.

We now split into three separate cases, depending on whether the $HHHHV$ in the cutting sequence is the first, second, or third chunk.Tables \ref{case1} through \ref{case3} show how the next three chunks play out in each case.

\begin{table}[h]
\begin{tabular}{lllll}
Encounter & Type & Collides With & Passes Through & Direction Afterwards \\
1         & H    & (0,1)H        &                & DR                   \\
2         & H    & (0,0)H        &                & UR                   \\
3         & H    &               & (0,1)H         & UR                   \\
4         & H    & (0,2)H        &                & DR                   \\
5         & V    & (1,1)V        &                & DL                   \\
6         & H    &               & (0,1)H         & DL                   \\
7         & H    &               & (0,0)H         & DL                   \\
8         & H    & (0,-1)H       &                & UL                   \\
9         & V    & (0,-1)V       &                & UR                   \\
10        & H    &               & (0,0)H         & UR                   \\
11        & H    &               & (0,1)H         & UR                   \\
12        & H    &               & (0,2)H         & UR                   \\
13        & V    & (1,2)V        &                & UL                  
\end{tabular}
\caption{\label{case1}
The case where the first chunk is $HHHHV$.
}
\end{table}

\begin{table}[h]
\begin{tabular}{lllll}
Encounter & Type & Collides With & Passes Through & Direction Afterwards \\
1         & H    & (0,1)H        &                & DR                   \\
2         & H    & (0,0)H        &                & UR                   \\
3         & H    &               & (0,1)H         & UR                   \\
4         & V    & (1,1)V        &                & UL                   \\
5         & H    & (0,2)H        &                & DL                   \\
6         & H    &               & (0,1)H         & DL                   \\
7         & H    &               & (0,0)H         & DL                   \\
8         & H    & (0,-1)H       &                & UL                   \\
9         & V    & (0,-1)V       &                & UR                   \\
10        & H    &               & (0,0)H         & UR                   \\
11        & H    &               & (0,1)H         & UR                   \\
12        & H    &               & (0,2)H         & UR                   \\
13        & V    & (1,2)V        &                & UL                  \\
\end{tabular}
\caption{\label{case2}
The case where the second chunk is $HHHHV$. Note that the only difference is swapping the order of encounters 4 and 5 and the direction the particle moves between them.
}
\end{table}

\begin{table}[h]
\begin{tabular}{lllll}
Encounter & Type & Collides With & Passes Through & Direction Afterwards \\
1         & H    & (0,1)H        &                & DR                   \\
2         & H    & (0,0)H        &                & UR                   \\
3         & H    &               & (0,1)H         & UR                   \\
4         & V    & (1,1)V        &                & UL                   \\
5         & H    & (0,2)H        &                & DL                   \\
6         & H    &               & (0,1)H         & DL                   \\
7         & H    &               & (0,0)H         & DL                   \\
8         & V    & (0,-1)V       &                & DR                   \\
9         & H    & (0,-1)H       &                & UR                   \\
10        & H    &               & (0,0)H         & UR                   \\
11        & H    &               & (0,1)H         & UR                   \\
12        & H    &               & (0,2)H         & UR                   \\
13        & V    & (1,2)V        &                & UL                  
\end{tabular}
\caption{\label{case3}
The case where the third chunk is $HHHHV$. The only difference from the second case is swapping the order of encounters 8 and 9 and the direction of motion between them
}
\end{table}

It turns out that all three cases play out very similarly, with just a couple of changes to the order in which particular walls are erased. At the end of the third chunk, all three are in the position shown in Figure \ref{caseconvergence}. The vertical walls erased are  (-1,-2)V, (-1,-1)V, (0,0)V, (0,1)V, (1,1)V, and (1,2)V. The horizontal walls erased are from (-1,-2)H through (-1,1)H and from (0,-1)H through (0,2)H. We now investigate what happens with repeated $HHHV$s from that position; when the particle passes through multiple erased horizontal walls, we combine into one line to save space.

 \begin{figure}[h]
\begin{center}
\includegraphics[height=7 cm]{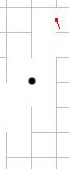}
\end{center}
\caption{\label{caseconvergence}
The state of the particle and its surroundings when the three different cases reconverge.
}
\end{figure}

\begin{table}[h]
\begin{tabular}{lllll}
Encounter & Type & Collides With & Passes Through     & Direction Afterwards \\
1         & H    & (0,3)H        &                    & DL                   \\
2-3       & HH   &               & (0,2)H, (0,1)H     & DL                   \\
4         & V    &               & (0,0)V             & DL                   \\
5-7       & HHH  &               & (-1,0)-(-1,-2)H    & DL                   \\
8         & V    & (-1,-3)V      &                    & DR                   \\
9         & H    & (-1,-3)H      &                    & UR                   \\
10-11     & HH   &               & (-1,-2)H, (-1,-1)H & UR                   \\
12        & V    &               & (0,-1)V            & UR                   \\
13-15     & HHH  &               & (0,0)-(0,2)H       & UR                   \\
16        & V    &               & (1,2)V             & UR                   \\
17        & H    & (1,3)H        &                    & DR                   \\
18        & H    & (1,2)H        &                    & UR                   \\
19        & H    &               & (1,3)H             & UR                   \\
20        & V    & (2,3)V        &                    & UL                   \\
21        & H    & (1,4)H        &                    & DL                   \\
22-23     & HH   &               & (1,3)H, (1,2)H     & DL                   \\
24        & V    &               & (1,1)V             & DL                   \\
25-27     & HHH  &               & (0,1)-(0,-1)H      & DL                   \\
28        & V    & (0,-2)V       &                    & DR                   \\
29        & H    & (0,-2)H       &                    & UR                   \\
30-31     & HH   &               & (0,-1)H, (0,0)H    & UR                   \\
32        & V    & (1,0)V        &                    & UL                   \\
33-35     & HHH  &               & (0,1)H-(0,3)H      & UL                   \\
36        & V    & (0,3)V        &                    & UR                   \\
37        & H    & (0,4)H        &                    & DR                   \\
38-39     & HH   &               & (0,3)H, (0,2)H     & DR                   \\
40        & V    &               & (1,1)V             & DR                   \\
41        & H    & (1,1)H        &                    & UR                   \\
42-43     & HH   &               & (1,2)H, (1,3)H     & UR                   \\
44        & V    &               & (2,3)V             &                     
\end{tabular}
\caption{\label{maintable}
The results of chaining together eleven $HHHV$ chunks from the position in Figure \ref{caseconvergence}.
}
\end{table}

Tracking the particle through the encounters in Table \ref{maintable} leaves it entering a column which, as far as we've tracked, is fresh; the only wall in the column $2 \leq x \leq 3$ which has been erased is (2,3)V, through which the particle is entering the column. We're therefore set as long as the particle has at least three $HHHV$ chunks lined up before the next $HHHHV$. But this just requires the $HHHHV$s to be spaced sufficiently far apart. 

We had at most 2 $HHHV$ chunks between the $HHHHV$ and the start of Table \ref{maintable}, and went through 11 chunks in the table; 2+11+3 = 16, so the lemma holds for any slope with at least 16 $HHHV$s between $HHHHV$s. This is equivalent to saying that if the cutting sequence is written as $(HHHV)^{a_1} H (HHHV)^{a_2} \ldots$, all the $a_i \geq 17$, which in turn is equivalent to saying that the continued fraction expansion of the slope starts
\begin{equation*}
s = 3 + \frac{1}{n + \ldots}
\end{equation*}
for some $n \geq 17$; this is obviously true iff $s \in (3, 3+\frac{1}{17}]$. 

To see that 17 is a strict lower bound on $n$ for the reorganization pattern to work, note that we need at least 14 $HHHV$ chunks to get through Table \ref{maintable} and one standard 3-style column, and if we have only 14 or 15 $HHHV$ chunks then the next $HHHHV$ chunk will happen earlier in its column; eventually, the $HHHHV$ chunk is the first chunk in its column and we have to have 16 $HHHV$ chunks to keep the process going.
\end{proof}

\begin{corollary}
\label{periodiconcestarted}
Suppose a particle with rational slope $s \in (3, 3+\frac{1}{17}]$ begins the pattern of behavior described in Lemma \ref{reorglemma}. Then unless it hits a corner or runs into a part of the plane with previously-cleared walls, it will dig a periodic tunnel with slope $1 + \frac{3s-9}{25-8s}$.
\end{corollary}
\begin{proof}
As long as we don't run into any corners or pre-cleared walls, the behavior described in Lemma \ref{reorglemma} will continue indefinitely. Rational slopes have periodic cutting sequences, so we can use our knowledge of that behavior to calculate the exact displacement per pass through the cutting sequence. As in the proof of Lemma \ref{reorglemma}, we reflect the plane so that the particle is moving up and to the right for the following calculations; this will not affect whether it tunnels or the slope of such a tunnel. Specifically, suppose $s = 3 + \frac{a}{b}$, where $a$ and $b$ are (not necessarily relatively prime) integers. Then there's a periodic cutting sequence with $3b+a$ $H$s and $b$ $V$s corresponding to slope $s$. 

Since every chunk is either $HHHHV$ or $HHHV$, noting that we have one $V$ either way and only the $H$s we must have $a$ HHHHV chunks and $b-a$ HHHV chunks. One $HHHHV$ chunk plus thirteen $HHHV$ chunks translate the particle by (2,3), per the calculations in the proof of Lemma \ref{reorglemma} (eleven $HHHV$s in Table \ref{maintable}, one $HHHHV$ and two $HHHV$ in the setup case). Three $HHHV$ chunks translate by (1,1), as we know from our study of slope 3. So our total translation is $a(2,3) + \frac{b-14a}{3}(1,1)$. If $b-14a$ isn't divisible by 3, this is going to spit out a fraction; that corresponds to a pass through the periodic cutting sequence leaving us with a different number of $HHHV$s done and therefore ready to start the next reorganization in a different subcase. Looping through the sequence three times (multiplying $a$ and $b$ by 3) will see our collisions behaving properly periodically, so we can afford to ignore this possibility.

Assuming $b-14a$ divisible by 3, we therefore have displacement $(2a +\frac{b-14a}{3}, 3a +\frac{b-14a}{3})$. Consistent displacement after a given number of collisions makes this a periodic tunnel, and we can use the displacement to see that the slope is
\begin{equation*}
\frac{b-5a}{b-8a} = \frac{1-5a/b}{1-8a/b} = \frac{1-5(s-3)}{1-8(s-3)} = 1 + \frac{3s-9}{25-8s}.
\end{equation*}
Note that this is a defined, continuous, and increasing function on the interval we're interested in.
\end{proof}

The analogous result for irrational slopes and proving that such particles do get started require a bit more work, which will be done in the next section.
 
Slope 3 is not the only place near which we see this kind of imitation-with-reorganization. As we saw in Figure \ref{reorghappens}, the same happens for slopes slightly larger than 33 (and, simulations suggest, for slopes slightly smaller). We also see similar behavior either side of $\frac{5}{3}$.

The neighborhood of slope 3 was was chosen for this detailed investigation over the others mentioned primarily because its tunnel, cutting sequence, and reorganization period are all very simple, but there are a couple of particularly interesting things about it. First is the fact that, per Figure \ref{reorgdoesnthappen}, the imitation-with-reorganization happens on one side of 3 but not the other. Secondly, stepping just past the  $3 \frac{1}{17}$ boundary for that behavior still seems to result in (almost-)periodic tunnels - see Figure \ref{oddperiodic}.

 \begin{figure}[h]
\begin{center}
\includegraphics[height=6 cm]{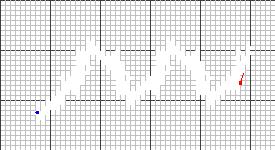}
\includegraphics[height=5 cm]{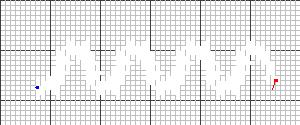}
\end{center}
\caption{\label{oddperiodic}
Starting in the center with slopes $3 \frac{1}{16}$ and  $3 \frac{1}{14}$. Both produce similar periodic tunnels, but not on the same model as slope 3 or $3+\epsilon$.
}
\end{figure}

These tunnels are harder to understand than the 3-with-reorganization tunnels, and the particle seems to stop making them somewhere between $3 \frac{1}{14}$ and $3 \frac{3}{40}$. Running simulations to approximate the exact transition point does not produce a nice obvious candidate.

\section{Continuity And Uncountable Tunnelling Slopes}

Although Lemma \ref{reorglemma} showed that rational slopes slightly bigger than 3 tunnel with reorganization once they get started, we still haven't quite shown that they necessarily start tunneling in the first place, or explained what happens to irrational slopes. This section will be devoted to fixing both holes and ultimately proving that for every starting point there is an interval on which every slope tunnels or hits a corner. We'll accomplish this by way of a more general discussion of the relationship between close starting conditions. Intuitively it seems obvious that close starting conditions should agree for a while, and agree more closely the closer the starting conditions; in fact, it's easy to construct an argument showing that this happens. 

\begin{lemma}
\label{continuitylemma}
Let P be a point in the interior of the unit square, $s$ be a real number, and $n$ be any integer such that a particle starting from P with slope $s$ does not hit a corner in the first $n$ collisions. Then there is a neighborhood $U$ of P such that starting a particle with slope $s$ anywhere in U results in the same first $n$ collisions, and a neighborhood $V$ of $s$ such that starting from P with any slope in V results in the same first $n$ collisions.
\end{lemma}
\begin{proof}
If the particle kept running in the same direction after colliding with a wall, rather than reflecting off, the result would be obvious, since moving the starting point a small amount or changing the direction a little can only have a tiny effect, and a small enough effect won't move any of the encounters corresponding to the first $n$ collisions. But if two particles both reflect off the same wall, the reflection doesn't change whether their next collisions agree, so even with reflection we still have particles agreeing until they've travelled far enough that they would encounter different walls by going in a straight line.
\end{proof}

The lemma is phrased in terms of varying P and $s$ separately because that makes it easy to derive a statement about the boundary of the optimal U and V; it's also easy to find a version which deals with varying P and $s$ simultaneously.

\begin{corollary}
\label{boundarycorners}
For $P$, $s$, and $n$ as above, and taking U and V to be the largest possible neighborhoods given by Lemma \ref{continuitylemma}, a particle starting from any point $Q \in \partial U$ with slope $s$ will hit a corner in the first $n$ collisions. Similarly, a particle starting from P with a slope $s' \in \partial V$ will hit a corner in the first $n$ collisions. 
\end{corollary}
\begin{proof}
Suppose there's some $Q$ in $\partial U$ such that this isn't true. Then there's a neighborhood $U'$ of Q agreeing on the first $n$ collisions by Lemma \ref{continuitylemma}; since $Q \in \partial U$, $U \cap U' \neq \emptyset$, so $U'$ agrees with P on the first $n$ collisions starting with slope $s$. By the maximality of U, we have $U' \subset U$, contradicting $Q\in \partial U$. The proof for V is identical. 
\end{proof}

\begin{corollary}
\label{strongercontinuity}
For $P$, $s$, and $n$ as above, there exist neighborhoods U and V of $P$ and $s$ such that all starting conditions in U$\times$V agree with $(P, s)$ on the first $n$ collisions.
\end{corollary}
\begin{proof}
If slopes $s_1$ and $s_2$ agree on the first $n$ collisions from a given point $Q$, then every slope $\alpha$ in $(s_1, s_2)$ also agrees, since the path of a particle with slope $\alpha$ is squeezed between the path with slope $s_1$ and the path with slope $s_2$. By Lemma \ref{continuitylemma}, there exists a neighborhood U$=(s_1,s_2)$, $s_1 < s < s_2$, in which every slope agrees with slope $s$ for the first $n$ collisions. Let $V_1$ ($V_2$) be the neighborhood of $P$ fixing the first $n$ collisions with slope $s_1$ ($s_2$). Then the result holds for $V = V_1 \cap V_2$.
\end{proof}

As discussed in Section 3, with a few minor exceptions where they correspond to the same starting point in the same cutting sequence, different starting conditions will eventually diverge. Still, knowing that sufficiently close starting conditions agree for a while is enough for us to extend Corollary \ref{periodiconcestarted} to irrational slopes and for us to prove that slopes slightly bigger than 3 can get the start they need to begin tunneling.

\begin{corollary}
\label{aperiodiconcestarted}
Suppose a particle with irrational slope $s \in (3, 3+\frac{1}{17}]$ begins the pattern of behavior described in Lemma \ref{reorglemma}. Then unless it hits a corner or runs into a part of the plane with previously-cleared walls, it will dig an aperiodic tunnel with slope $1 + \frac{3s-9}{25-8s}$.
\end{corollary}
\begin{proof}
Since particles with irrational slope have aperiodic cutting sequences, they can't dig periodic tunnels; we just need to show that the behavior described in Lemma \ref{reorglemma} does dig a tunnel with an appropriate slope. We already know this is true for particles of rational slope under the same assumptions by Corollary \ref{periodiconcestarted}, so the basic idea here is to squeeze our $s$ between better and better rational approximations. Again, we can assume that the particle is headed up and to the right because reflection does not change the existance or slope of a tunnel. 

By Lemma \ref{continuitylemma}, there exists a neighborhood of $s$ on which every slope agrees with $s$ until after $s$ has begun (what we want to show is) tunneling. Pick rational $p,q$ such that $p<s<q$, they both satisfy Corollary \ref{periodiconcestarted} and they both agree with $s$ until after they've begun tunneling. Then we can choose balls for the definition of $p$ and $q$'s tunnels such that every point where they disagree with $s$ is in the appropriate bands rather than the balls.

Now, a particle behaving as in Lemma \ref{reorglemma} will travel with slope 1 when not reorganizing (i.e. no recent $HHHHV$s) and with slope 1.5 over a reorganizational patch (i.e. in the vicinity of an $HHHHV$). The overall slope of the tunnel therefore increases the more $HHHHVs$ you have in the cutting sequence, justifying our observation that for rational slopes the tunnel's slope was an increasing function. We can therefore say that, since $s<q$, in any given column the walls cleared by the $s$-particle are level with or lower than the walls cleared by the $q$-particle, so the walls cleared by the $s$-particle outside of the ball are bounded above by the upper boundary of the $q$-band. Similarly, the walls cleared by the $s$-particle outside of the ball are bounded below by the lower boundary of the $p$-band. But we know from Corollary \ref{periodiconcestarted} the slopes of these boundaries are given by a continuous function of $p$ (or $q$), so letting $p$ and $q$ converge to $s$ the boundaries will converge to form a band with slope given by that function. 
\end{proof}

We now know that both rational and irrational slopes will tunnel if they run into the starting conditions for Lemma \ref{reorglemma}, which means we are ready to prove our main result.

\begin{theorem}
\label{everypointtunnels}
For every point P in the interior of the unit square, there is some $\epsilon$, $0 < \epsilon \leq \frac{1}{17},$ such that for every slope $s$ in $[3,3+\epsilon]$ a particle starting at P with slope $s$ either hits a corner or clears a tunnel, with the tunnel's slope being $1 + \frac{3s-9}{25-8s}$.
\end{theorem}
\begin{proof}
We know from Corollaries \ref{periodiconcestarted} and \ref{aperiodiconcestarted} that if non-3 slopes in the given range start tunneling and have an untouched plane ahead of them they keep tunneling with the stated slope. We've already seen that when slope 3 tunnels, it digs a tunnel whose slope is 1, so the slope formula works there too. It's easy to check that particles with slope 3 quickly start tunneling into an untouched plane for any starting position except the line of slope 3 going from (0,0) to (1,1) and wrapping around twice, for which it quickly hits a corner. For P not on that line, the result follows immediately by Lemma \ref{continuitylemma}. For P on the line, we note that a slope $s>3$ will have the same first few collisions as a particle starting slightly above P with slope 3; taking $s$ close enough to 3 and the other point close enough to P, the two must agree long enough to start tunneling. 
\end{proof}

\begin{corollary}
\label{everypointinfiniteperiod}
For every point P in the interior of the unit square, there exist slopes such that a particle starting at point P with the given slope tunnels periodically with arbitrarily large period.
\end{corollary}
\begin{proof}
The length of the period of a periodic tunnel-with-reorganization depends on when the spacing of the $HHHHV$s begins to repeat, hence on the length of the period of the cutting sequence. One can find cutting sequences in the appropriate region of arbitrary length by choosing a sufficiently complicated continued fraction, or just a slope close enough to 3.
\end{proof}

 \begin{corollary} 
 \label{uncountabletunnels}
 For any P in the interior of the unit square, there are uncountably many slopes which produce tunnels starting at P.
 \end{corollary}
 \begin{proof}
 Reflection doesn't change whether a particle hits a corner, so the set of slopes which hit a corner is the same as the set of slopes which cross elements of the integer lattice starting from P, which is countable. The slopes in the real interval $[3,3+\epsilon]$ not hitting corners from P, which we know produce tunnels, are therefore uncountable in number.
 \end{proof}

 \section{Slope 146 And Delayed Tunneling}
 
  \begin{figure}[h]
\begin{center}
\includegraphics[height=8 cm]{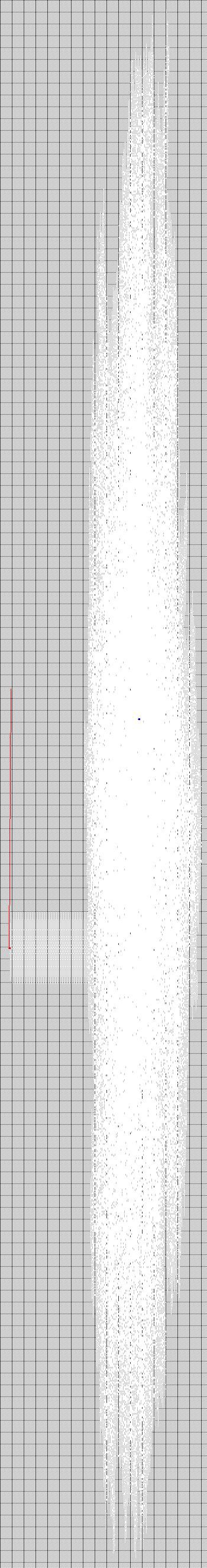}
\includegraphics[height=7 cm]{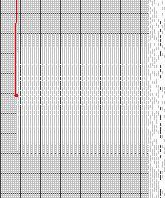}
\end{center}
\caption{\label{slope146}
A particle starting with slope 146 from the center clears a skinny oval for a long time, but somewhere between 150,000-160,000 collisions it begins a periodic tunnel going straight left. The full cleared region is on the left, a zoom of the tunnel on the right.
}
\end{figure}

\begin{figure}[h]
\begin{center}
\subfigure[ ]{\includegraphics[height=12 cm]{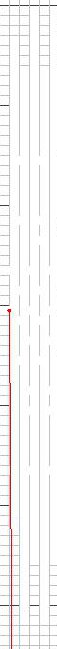}} \hfill
\subfigure[ ]{\includegraphics[height=12 cm]{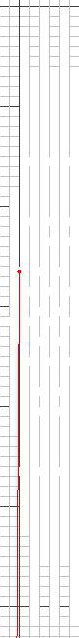}} \hfill
\subfigure[ ]{\includegraphics[height=12 cm]{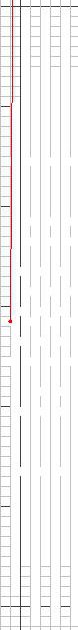}}
\end{center}
\caption{\label{146close}
Three different stages of the periodic tunnel formed by slope 146.
}
\end{figure}
 
 While it would be convenient if every set of starting conditions which produced a tunnel started tunneling quickly enough to be hand-checkable, there's no reason to believe that should be true and on simulation we see quite the contrary. The most spectacular example of a slope which can tunnel but only after a long time setting up is 146. Figure \ref{slope146} shows what happens to a particle starting from the center with slope 146; it turns out that this is far from the the longest it can take to get started.

Discussing its range of behaviour will need some way of labeling the various starting locations:

\begin{definition}
\label{startnumbering}
There are 147 meaningfully different starting locations for a particle with slope 146, since it has a cutting sequence of length 147. We can see those locations on the square by taking a line with slope 146 from (0,0) and wrapping it around repeatedly until it hits (1,1). The line will divide the square into 147 regions. We number the region on the far left 0, the next region 1, and so on til region 146 on the right; looking at the cutting sequence, these numbers correspond to the number of horizontal encounters between the next encounter and the previous vertical encounter (so where the particle just had a vertical encounter, it's numbered 0).
\end{definition}
 
The author tested every starting location and recorded the full results  \href{https://docs.google.com/spreadsheet/ccc?key=0AhFI-rEFDT-rdERtbVY1ekFVTWlPRjhCNjdpLXRReVE&usp=sharing}{in a public Google spreadsheet}. The number of steps displayed for each start to begin tunneling is approximate, but is the lowest possible answer with the given number of significant digits. Every starting location eventually begins tunneling, but it often takes a long time: the mean number of collisions needed is around 1.42 million, and the slowest starting location only shows signs of tunneling after 7.18 million collisions.

The reason slope 146 takes longer to start tunneling than anything we've looked at before is that the periodic tunnel it creates is more complicated. Rather than bouncing back and forth across one column until it's ready to move to the next, the particle ducks back into the previous column partway through. Figure \ref{146close} shows this process. On the left, the particle is entering a new column for the first time; we can see that the column isn't quite fresh and, more importantly, the column it's leaving hasn't been fully cleared yet. In the middle, the particle is heading back into the incomplete column. And on the right, it's entering another new column, with the original column now cleared like the ones before it and the middle column partly cleared; as far as the particle's immediate surroundings go, the right-hand picture is a vertical flip of the left-hand picture. Getting to the perfect vertical flip takes 61 collisions or 1323 encounters; the tunnel is therefore periodic with period 122 collisions / 2646 encounters.

This need to return to the previous column means that the particle cannot start tunneling without a fairly precise set of conditions - it needs the previous column to have exactly the right horizontal walls cleared \emph{and} to have the vertical walls it passes through or hits in the appropriate state. It is therefore unsurprising that, while these conditions do always occur, they often take a while to arise. It is an open question whether there is a possible periodic tunnel in which the particle works on more than two columns simultaneously - and, if such a tunnel exists, whether the particle actually tunnels from every starting location.
 
 \section{Non-Tunneling Behaviour}

More basically, it's reasonable to ask whether the particle always has to produce some kind of tunnel. There's no obvious reason that the particle should fall into any kind of pattern instead of clearing the entire plane, and simulations suggest that for a lot of starting conditions the particle does just keep clearing wider and wider areas around its starting point.  However, proving that this continues indefinitely has so far been an intractable problem - the particle could eventually hit some specific set of conditions causing it to go periodic, and we saw in the previous section that this can take a long time to happen for sufficiently specific conditions (See Figure \ref{slope146}). 

 \begin{figure}[h]
\begin{center}
\includegraphics[height=6 cm]{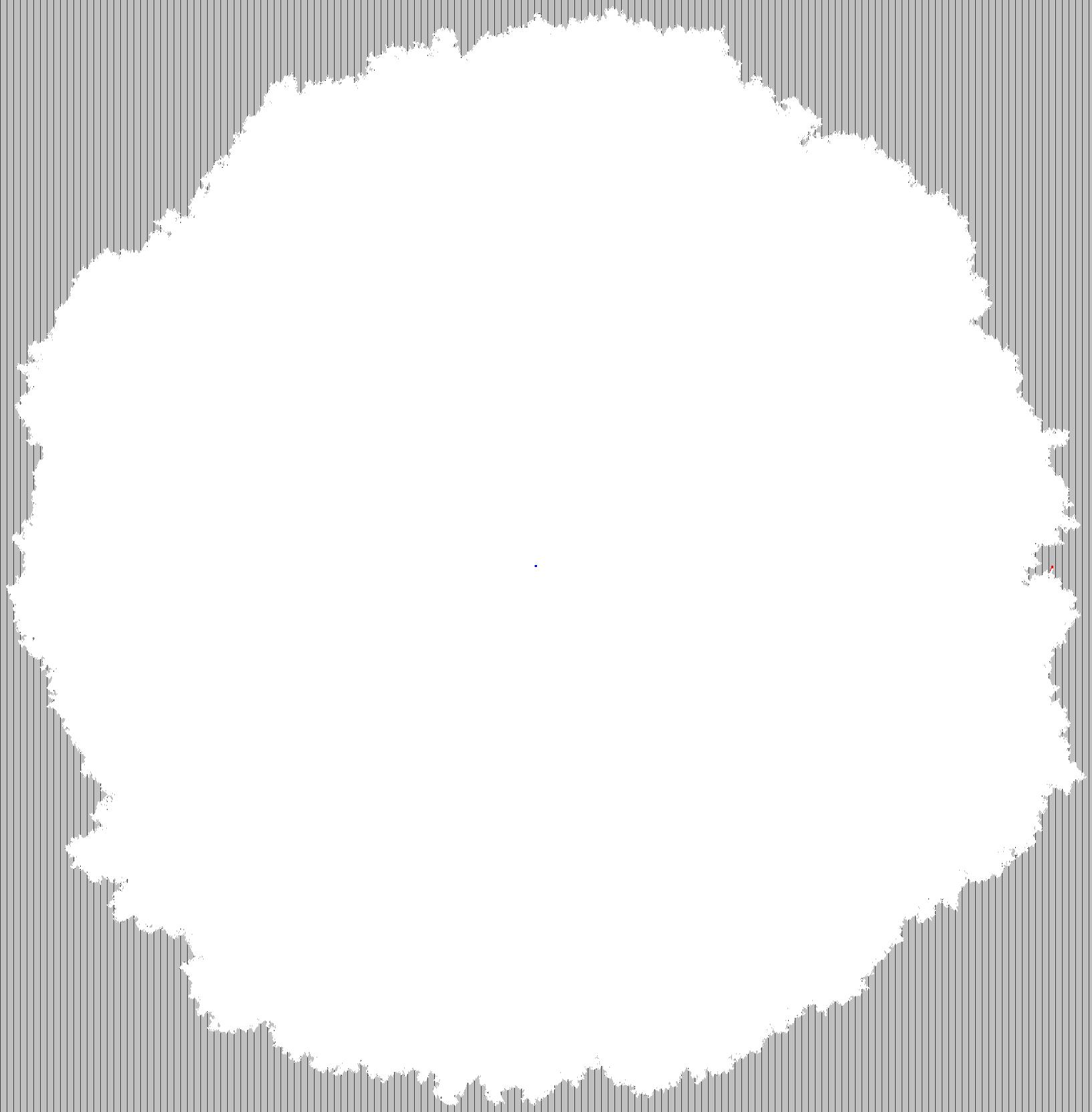}
\includegraphics[height=6 cm]{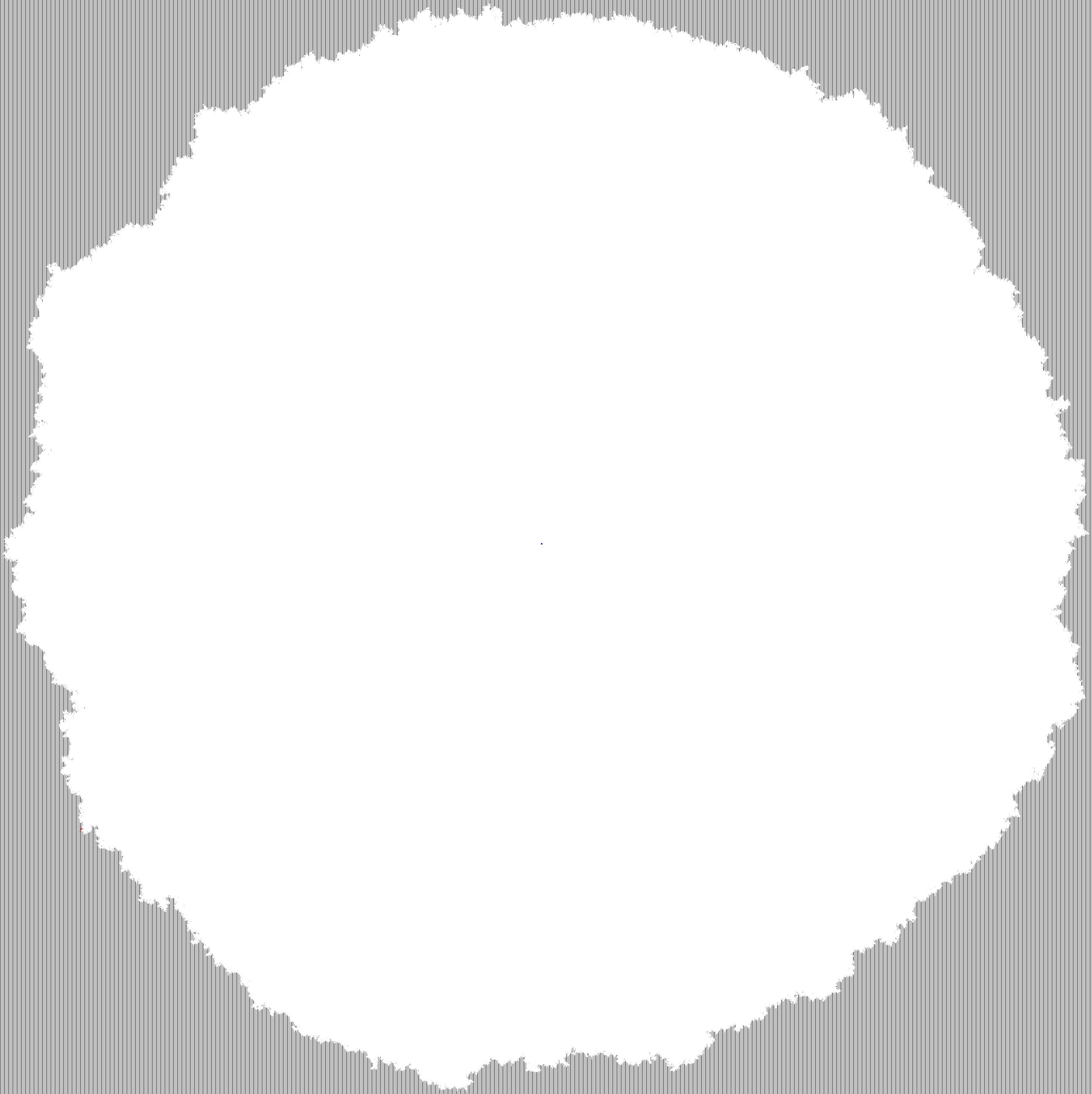}
\end{center}
\caption{\label{2circle}
A particle with slope 2 from the center of the square seems to just keep clearing a larger and larger roughly circular blob. The figure on the left is after four million collisions, the figure on the right after ten million.
}
\end{figure}

Still, given the length of time which slope 2 goes without clearing anything vaguely resembling a tunnel, the following conjectures seem plausible, in increasing order of strength:

\begin{conjecture}
\label{clearing1}
There exist starting conditions for which the particle passes within some bounded distance of every wall in the plane. 
\end{conjecture}

\begin{conjecture}
\label{clearingstrong}
There exist starting conditions for which the particle eventually hits every wall in the plane. In particular, this happens for particles with slope 2 at any point in the square from which they don't hit a corner.
\end{conjecture}

\section{Wedge-Shaped Bombs}

Most of this paper has focused on the simple single-wall bomb, but there are some types of behavior which are more easy to see when looking at other bombs. Bidirectional tunnels seem more common once we start looking at more complex bombs, and we've only seen the same initial direction produce meaningfully distinct tunnels once nontrivial bombs get involved. To give a feel for how this can play out, we'll investigate bombs in the shape of a triangular wedge growing away from the point of impact. We define a winged wedge of size $n$ according to the picture:

\begin{figure}[h]
\begin{center}
\includegraphics[height=4 cm]{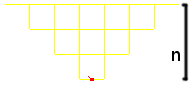}
\end{center}
\caption{\label{wingedwedgesample}
A winged wedge of size $n$. The particle clears one horizontal wall in the first row, three in the second, and so on up til $(2n+1)$ in row $(n+1)$, along with every vertical wall which touches cleared horizontal walls at both ends. Equivalently, it clears an isosceles triangle of height $n$ squares, plus one edge at each far corner.
}
\end{figure}

An unwinged wedge of size $n$ is the same shape minus the two edges at the far corners. The two types of wedge behave very similarly; we treat the winged wedge as the default because, when the two diverge, its behavior has the more interesting structure. 

The obvious question to ask, when working with bombs whose size we can vary, is whether there is any predictable relationship between a slope and which wedge sizes cause it to tunnel. Simulation suggests that there is some relationship:

\begin{conjecture}
\label{sdividesn+1}
Suppose we start a particle with integer slope $s>1$ and a wedge (either type) of size $n$, from a point in the square which will not cause it to hit a corner. If $s|(n+1)$, the particle will eventually tunnel.
\end{conjecture}

\begin{figure}[h]
\begin{center}
\includegraphics[height=6 cm]{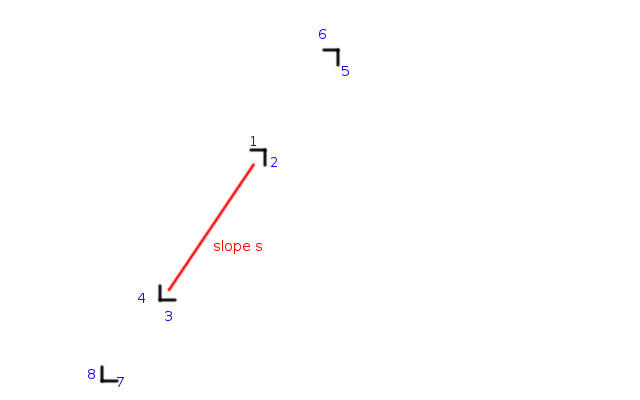}
\end{center}
\caption{\label{sdividesn+1pic}
The setup which causes a particle to tunnel with period 4, assuming that $s|(n+1)$ and it doesn't find something else to do first.
}
\end{figure}

The reason this seems to hold true lies in the setup shown in Figure \ref{sdividesn+1pic}. If we have a clear space between two solid opposite corners, the slope of a path between them is $s$, and the particle is at a point in its cutting sequence to hit both walls of one corner in quick succession (HV in the picture, but VH also works), then its next two collisions will be with the two walls of the other corner. Since $s>1$, the particle heads back through wall 1 and into the triangle cleared by hitting wall 1. Hitting wall 1 at height $h$ cleared horizontal walls through height $n+h$, so the next solid horizontal wall (wall 5) is at height $(n+1)+h$, which the particle will encounter at the same point in its cutting sequence as wall 1 if $s|(n+1)$. Assuming that wall 6 is also solid, the particle will treat 5 and 6 the same way it treated 1 and 2, heading off towards 7 and 8 \ldots

So we see that if the particle ever hits two corners in succession, while having solid corners appropriately spaced further along the line between them, it will create a bidirectional tunnel with period 4. The problem in proving this is that, since the corners need a space whose diagonal has slope $s$ between them, the setup cannot come around immediately. Sometimes it happens quickly; sometimes it takes a long time; sometimes, as Figure \ref{differenttunnels} demonstrates, the particle manages to settle into an entirely different periodic tunnel before the conditions arise to start this one. The conditions for this tunnel are simple enough that it seems likely they will always arise unless preempted by another tunnel, but we do not yet have a proof.

\begin{figure}[h]
\begin{center}
\includegraphics[height=8 cm]{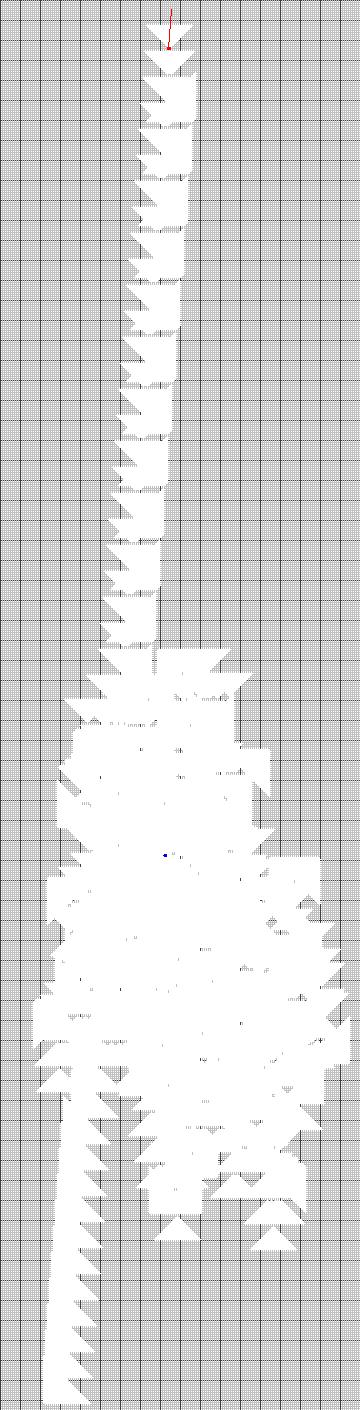}
\includegraphics[height=8 cm]{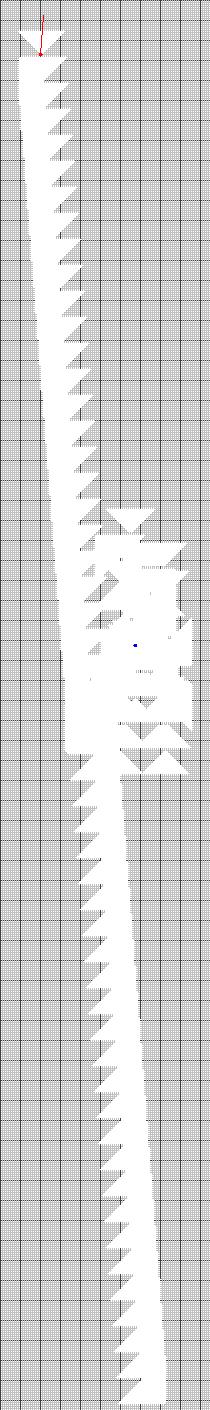}
\end{center}
\caption{\label{differenttunnels}
Two different tunnels dug by a particle with slope 13 and an unwinged wedge bomb of size 12. The one on the left started on the left-hand wall; the one on the right started one stage further into the cutting sequence, and is the type of tunnel discussed in Conjecture \ref{sdividesn+1},
}
\end{figure}

\begin{remark}
\label{varyingsetup}
Figure \ref{differenttunnels} is the first case we've discussed where the same bomb, initial direction, and initial wall configuration can produce different tunnels depending on the starting location. We're not aware of any case where this happens for the single-wall bomb, though we see no compelling argument against the possibility. Since this direction and bomb shape can be attracted to multiple behavior patterns, this example also shows that the same bomb, direction, and starting location can produce different tunnels depending on the initial wall configuration - just lay the walls out as if the particle were partway through the appropriate tunnel at the given spot in its cutting sequence. We saw while proving Theorem \ref{everypointtunnels} that  varying the direction produces similar-but-different tunnels, and we'll see shortly that starting with slope 3 on the left-hand wall produces a variety of tunnels depending on the bomb shape. Thus we see that for all four of our choices when creating the system, varying the one while holding the other three constant can make the outcome fluctuate between multiple different tunnels.
\end{remark}

Conjecture \ref{sdividesn+1} would imply that every integer slope tunnels infinitely often for either wedge shape, since it's not hard to show that $s=1$ and $s=0$ tunnel frequently. The slope which tunnels most frequently, though, seems to be slope 3. In fact, starting from the left-hand wall, we can prove slope 3 tunnels for every winged wedge size not covered by Conjecture \ref{sdividesn+1}, suggesting that it probably tunnels for every winged wedge.

\begin{remark}
\label{chartskey}
Throughout the proofs of the next two results:
\begin{itemize}
\item Walls are indexed by their lower (if vertical) or leftmost (if horizontal) endpoint, and labeled H or V as appropriate.
\item The particle is at stage $m$ in its cutting sequence when it has encountered $m$ horizontal walls since the last vertical encounter, and therefore has $(3-m)$ horizontal encounters before the next vertical encounter.
\item The stage and direction given in each step are for immediately after the collision or encounter in question.
\end{itemize}
We also observe that if the particle hits wall $(a,b)$ going up, with wedge size $n$, the solid horizontal walls at the top have coordinates $(x, b+n+1)$ for some $x$; the solid vertical walls have coordinates $(x, b+n)$; the solid vertical walls along either diagonal side of the triangle have coordinates $(-m+a, m+b-1)$ or $(m+a+1, m+b-1)$ for some integer $m \leq n$; and the solid horizontal walls along the diagonal sides have coordinates $(m+a, m+b-1)$ or $(-m+a, m+b-1)$. One can find the corresponding boundaries for other directions similarly, but this is the direction which is most frequently relevant.
\end{remark}

\begin{lemma}
\label{s3n0mod3}
A particle with slope 3 and a wedge-shaped bomb of size $3k$, $k \geq 1$, tunnels horizontally with period 14 when started from the left-hand wall.
\end{lemma}
\begin{proof}
We track one complete period in Table \ref{0mod3table}

\begin{table}[h]
\begin{tabular}{llll}
Step & Encounter and Type & Stage & Direction \\
1    & (0,1)H             & 1     & DR        \\
2    & (0,0)H             & 2     & UR        \\
3    & (k+1, 3k+1)V       & 0     & UL        \\
4    & (k, 3k+2)H         & 1     & DL        \\
5    & (-k, -3k-1)V       & 0     & DR        \\
6    & (-k, -3k-1)H       & 1     & UR        \\
7    & (2k+1. 6k+3)H      & 2     & DR        \\
8    & (4k+1, 4)H         & 1     & UR        \\
9    & (4k+1, 5)H         & 2     & DR        \\
10   & (5k+2, -3k+3)V     & 0     & DL        \\
11   & (5k+1, -3k+3)H     & 1     & UL        \\
12   & (2k+2, 6k+2)V      & 0     & UR        \\
13   & (2k+2, 6k+3)H      & 1     & DR        \\
14   & (6k+2,-6k+2)H      & 2     & UR        \\
15   & (8k+2, 1)H         & 1     & DR        \\
16   & (8k+2, 0)H         & 2     & UR       
\end{tabular}
\caption{\label{0mod3table}
The first sixteen steps of a particle with slope 3 and wedge size $3k$, $k \geq 1$, starting from the left-hand wall (i.e. stage 0).
}
\end{table}

Comparing steps 15 and 16 to steps 1 and 2, we see that the particle has translated $(8k+2, 0)$ while retaining the same direction and stage of the cutting sequence. The only walls which have been erased to the right of the particle's current location are those erased as a consequence of steps 15 and 16. Th only time the particle went left of $x=0$ in the first fourteen steps was for steps 5 and 6. We can verify that the corresponding walls at $(7k+2, -3k-1)$ are both still solid: the only steps erasing walls anywhere near are steps 11 (compared to which the walls are $2k+1$ over horizontally, 4 down, which is too far wide for $k > 1$ and too far down for $k = 1$) and 14 (higher for$ k > 1$, equal + to the right for$ k = 1$). We can also tell that the paths to both walls will be empty thanks to step 16. It follows that, since everything to the right is identical to the conditions after step 2 and so are the relevant conditions to the left, the particle tunnels with period 14. 
\end{proof}

\begin{figure}[h]
\begin{center}
\includegraphics[height=6 cm]{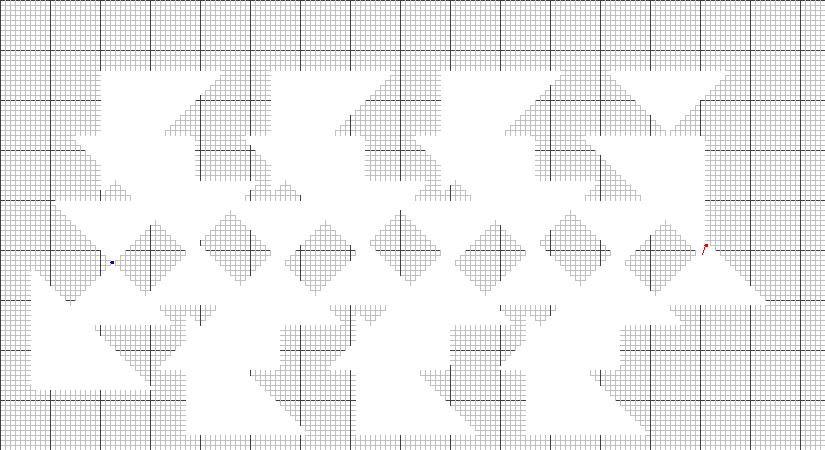}
\end{center}
\caption{\label{s3n12}
The periodic tunnel dug by a particle with slope 3 and a wedge whose size is divisible by 3 (in this case, a winged wedge of size 12)
}
\end{figure}

For the winged wedge, in particular, slope 3 displays even more structure in simulation.

\begin{theorem}
\label{s3n1mod3}
Let $n$ be an integer such that $n = 1$ mod $3$ and $p$ be the largest integer such that $2^p | n+2$. Then a particle started from the left-hand wall with slope 3 and a winged wedge bomb of size $n$ tunnels with period $6+2p$.
\end{theorem}

\begin{proof}
We'll handle the case $p=0$ separately. If $p=0$, then $n=1$ mod $6$ and we can just set $n=6k+1, k\geq0$. The first eight steps of that case are in Table \ref{1mod6table}.

\begin{table}[h]
\begin{tabular}{llll}
Step & Encounter and Type & Stage & Direction \\
1    & (0,1)H             & 1     & DR        \\
2    & (0,0)H             & 2     & UR        \\
3    & (2k+1, 6k+3)H      & 2     & DR        \\
4    & (3k+2, 3k+1)V      & 0     & DL        \\
5    & (3k+1, 3k+1)H      & 1     & UL        \\
6    & (0, 12k+5)H        & 2     & DL        \\
7    & (-k, 9k+3)V        & 0     & DR        \\
8    & (-k, 9k+3)H        & 1     & UR       
\end{tabular}
\caption{\label{1mod6table}
The first sixteen steps of a particle with slope 3 and wedge size $3k$, $k \geq 1$, starting from the left-hand wall (i.e. stage 0).
}
\end{table}
After step 8, the particle will pass through horizontal wall $(0, 12k+4)$ at stage 2, heading up and to the right with the only walls cleared above it those resulting from hitting $(0,12k+5)$ while moving up. This is exactly the situation after step 2, translated $12k+4$ squares straight up, so the system is periodic with period 6 as desired.

If $p\geq 1$, we set $n=3(2k+1)2^p-2$ and follow Table \ref{1mod3maintable}.

\begin{table}[h]
\begin{tabular}{llll}
Step & Encounter and Type  & Stage & Direction \\
1    & (0,1)H         & 1     & DR        \\
2    & (0,0) H    & 2     & UR        \\
3    & $((2k+1)2^p, 3(2k+1)2^p)$ H & 2     & DR        \\
4    & $((2k+1)(2^p+2^{p-1}), 3(2k+1)(2^p-2^{p-1})$H            & 2     & UR        \\
5    & $((2k+1)(2^p+2*2^{p-1}),  3(2k+1)2^p)$H                                                            & 2     & DR        \\
...  & ...                                                                                                                            & ...   & ...       \\
7    & $((2k+1)(2^p + 2*2^{p-1} + 2*2^{p-2},  3(2k+1)2^p)$ H              & 2     & DR        \\
...  & ...                                                                                                  & ...   & ...       \\
$2p+3$ & $((2k+1)(2^p + 2*2^{p-1} + ... + 2p^0),  3(2k+1)2^p)$ H               & 2     & DR        \\
eq.  & $((2k+1)(3*2^p-2), 3(2k+1)*2^ p)$   H                                                             & 2     & DR        \\
$2p+4$ &    $( 3(2k+1)*2^p-3k-1,  3(2k+1)2^p-3k-2)$V & 0     & DL        \\
$2p+5$ & $(3*2^p(2k+1)-3k-2,  3(2k+1)2^p)-3k-2)$ H                                                             & 1     & UL        \\
$2p+6$ & $(2*2^p(2k+1)-4k-2, 2* 3(2k+1)2^p-1) $H                                                             & 2     & DL        \\
$2p+7$ & $(2^p(2k+1)-3k-1,  3(2k+1)*2^p+3k$)V      & 0     & DR        \\
$2p+8$ & $(2^ p(2k+1)-3k-1, 3(2k+1)*2^p)+3k)$H                                                          & 1     & UR        \\
thru & $(2*2^p(2k+1)-4k-2, 2* 3(2k+1)*2^p-2)$H                                                         & 2     & UR       
\end{tabular}
\caption{\label{1mod3maintable}
The first $2p+8$ steps of a particle with slope 3 and wedge size $3(2k+1)2^p-2$, starting from the left-hand wall.
}
\end{table}

After step 4, the pattern of cleared and solid walls up and to the right is exactly as if the particle were headed up into a wedge of size $3(2k+1)2^{p-1}-2$, created by hitting the horizontal wall at $((2k+1)*(2^p+2^{p-1}), 3(2k+1)*(2^p-2^{p-1}+1)$ - in other words, as if we'd just finished step 2 with $p$ reduced by one. Because of this, induction on $p$ shows that we will be bouncing up and down while going to the right through step $2p+3$, after which continuing the pattern would imply a wall with fractional coordinates. Instead, the particle hits a vertical wall, then a horizontal wall, starts heading up and to the left, and passes through horizontal wall $(2*2^p(2k+1)-4k-2, 3(2k+1)*2^p)$ at cutting sequence stage 3 between collisions $(2p+5)$ and $2p+6$. After step $2p+8$, the encounter which is called out with its own line is exactly step 2 translated by $(2*2^p(2k+1)-4k-2, 2* 3(2k+1)*2^p-2)$; the same stage, the same direction, and once again the only walls cleared above are the wedge based on the wall immediately above. Since nothing in steps $3$ through $2p+8$ went below $y=0$, this is enough to show that the system is periodic with period $2k+6$ and displacement $(2*2^p(2k+1)-4k-2, 2* 3(2k+1)*2^p-2)$.
\end{proof}

\begin{figure}[h]
\begin{center}
\includegraphics[height=6 cm]{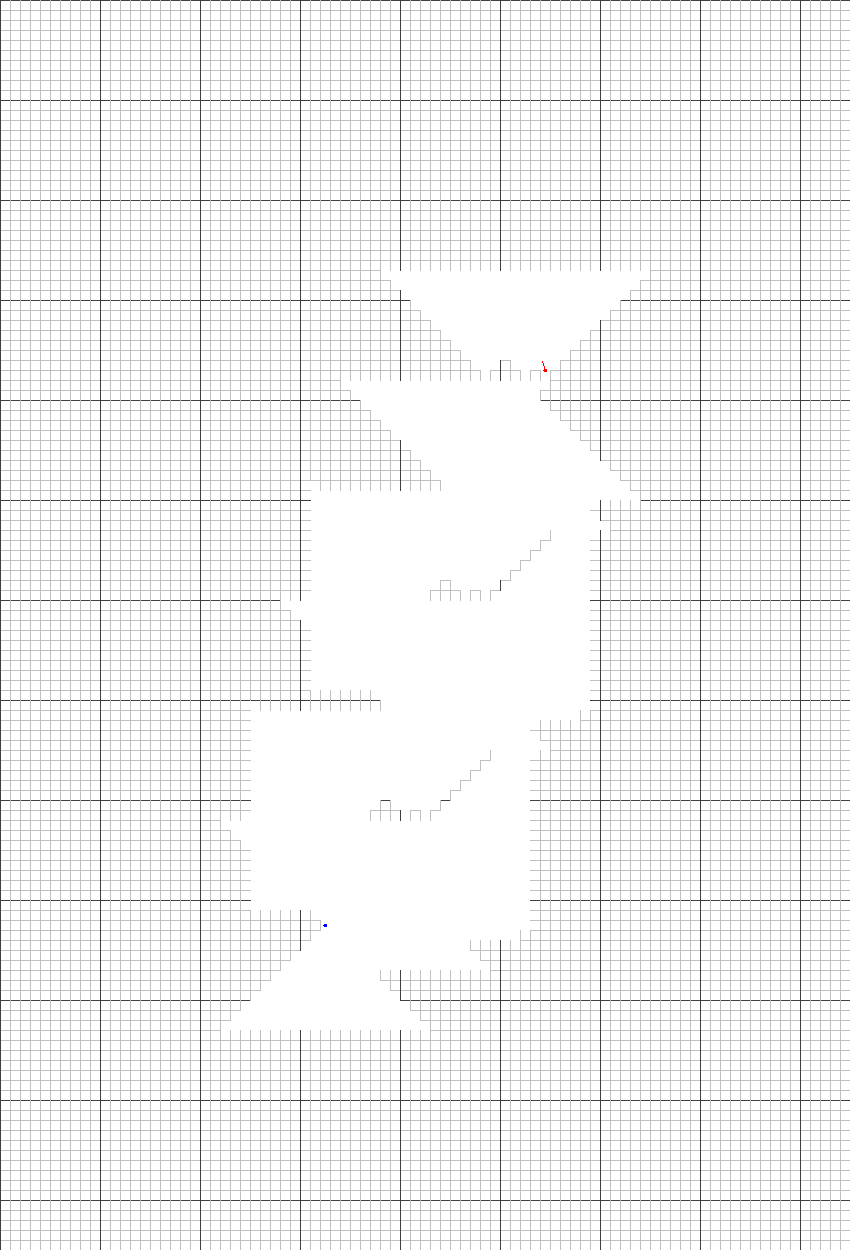}
\end{center}
\caption{\label{s3n10}
The periodic tunnel dug by a particle with slope 3 and a winged wedge bomb of size 10.
}
\end{figure}

This theorem is the only result in this section where we have had to specify that we're talking about winged, rather than unwinged, wedges. A natural question is to what extent the same holds for unwinged wedges.

\begin{corollary}
Let $n=3(2k+1)2^p-2$. A particle starting from the left-hand wall with slope 3 and an unwinged wedge bomb of size $n$ tunnels with period $2p+6$ if $k>0$.
\end{corollary}
\begin{proof}
The one possibly pertinent difference between a winged wedge of size $3(2k+1)2^p-2$ heading up from horizontal wall $(0,1)$ and an unwinged wedge of the same size is that for the unwinged wedge, the wall $(3(2k+1)2^p-2,3(2k+1)2^p-1)$ is solid. If $p>0$ and $k=0$, this is the wall immediately below the one hit in step $(2p+3)$, causing the particle to deviate from the periodic tunnel and invalidating the above proof. Indeed, an unwinged wedge of size 22 ($k=1, p=3)$ does not seem to result in any sort of tunnel. For wedge sizes where $k\geq1$, the proof holds. If $p=0$, the horizontal wall affected is $(6k+1, 6k+2)$, which is blocking the path taken for step 3 iff $k=0$.
\end{proof}

\begin{corollary}
\label{infinitebombs}
Let $q\geq 6$ be an even integer. Then there exist infinitely many winged wedges and infinitely many unwinged wedges causing a particle with slope 3 to tunnel with period $q$ when started from the left-hand wall.
\end{corollary}

\section{Open problems}

The biggest open problem, which we've already discussed, is the question of whether there is some set of starting conditions for which the system never tunnels. We have a clear candidate in slope 2 and the simple bomb, but as of yet no idea how to prove that it never tunnels. It is also unknown whether there is any structured form of long-term behavior besides tunnels - simulations only produce tunnels and big blobs, and it is difficult to imagine what non-tunneling structured behavior might consist of, but again that is not a proof. 

We have done very little with varying the starting wall configuration; it would be interesting to know whether there is some slope which digs different tunnels with the single-wall bomb depending on the starting walls.

We have also only scratched the surface of the vast domain of possible bombs. It remains to be seen whether there are other scaleable sets of bombs which cause some nontrivial slope to tunnel for every possible bomb size, as we conjecture the winged wedges do for slope 3.

There are a number of possible variants to the system. A related family of problems involves looking at what happens when the particle clears away two-dimensional obstacles instead of one-dimensional; see \cite{bressaud} for the most recent research along these lines. Even sticking with one-dimensional obstacles, one could look at a different set of walls . Any regular grid would do nicely, and the triangular and hexagonal grids seem the most obvious next step. However, they would suffer from the problem of not being able to rely on conventional cutting sequences absent some clever new approach; our Lemma \ref{cuttingsequence} depended on being able to label the grid in a way which was invariant under reflection across every edge, which is not a property it shares with the other grids. More broadly, if one doesn't care about being able to support a wide variety of bombs, any infinite collection of walls can be bounced off and erased one at a time. 

Another variant, whose two-dimensional analogue Bressaud and Fournier also study, is to bound the region in which the particle can bounce by adding some uneraseable walls; for instance, we might specify that the particle has to remain in the part of $R^2$ for which $y \geq 0$, and reflects off the x-axis whenever it comes down there. With a sufficiently constrained region, whether the particle tunnels becomes of less interest than how quickly it tunnels and whether it clears all the removable walls in its direction of motion.

\appendix
\section{Integer Slope Simulations}

In order to get some kind of sense of how many slopes tunnel and how many clear big blobs, and how long it can take for a tunneling slope to start tunneling, the author has simulated particles starting at the center of the square with every even slope through 200 (odd slopes starting at the center of the square eventually hit corners). The full data is available \href{https://docs.google.com/spreadsheet/ccc?key=0AhFI-rEFDT-rdHk0Umstb09HUVNrM2g3ajlGNTl0VkE&usp=sharing}{online}; roughly speaking, half of the integer slopes tested start tunneling within the first 300,000 collisions, most of them doing so in the first 10,000. The same spreadsheet also includes tests of the first few integer slopes with various wedge sizes, which inspired the results in Section 10.

\end{document}